\documentclass[leqno,12pt]{amsart} 
\usepackage[top=20truemm, bottom=20truemm, left=30truemm, right=30truemm]{geometry}
\usepackage{amssymb}
\usepackage{amsmath}
\usepackage{amsthm}
\usepackage{amscd}
\usepackage[dvips]{color}
\usepackage{mathtools}
\usepackage{ascmac}
\makeatletter
\newsavebox{\@brx}
\newcommand{\llangle}[1][]{\savebox{\@brx}{\(\m@th{#1\langle}\)}%
  \mathopen{\copy\@brx\kern-0.5\wd\@brx\usebox{\@brx}}}
\newcommand{\rrangle}[1][]{\savebox{\@brx}{\(\m@th{#1\rangle}\)}%
  \mathclose{\copy\@brx\kern-0.5\wd\@brx\usebox{\@brx}}}
\makeatother

\theoremstyle{plain} 
\newtheorem{theorem}{\indent\sc Theorem}[section]
\newtheorem{lemma}[theorem]{\indent\sc Lemma}
\newtheorem{corollary}[theorem]{\indent\sc Corollary}
\newtheorem{proposition}[theorem]{\indent\sc Proposition}

\theoremstyle{definition} 
\newtheorem{definition}[theorem]{\indent\sc Definition}

\newcommand{\delbar}{\overline{\partial}}

\newcommand{\vp}{\varphi}
\newcommand{\ve}{\varepsilon}

\newcommand{\U}{\mathcal{U}}
\newcommand{\R}{\mathbb{R}}
\newcommand{\C}{\mathbb{C}}
\newcommand{\D}{\mathbb{D}}
\newcommand{\Z}{\mathbb{Z}}

\newcommand{\twopii}{2\pi\sqrt{-1}}

\newcommand{\one}{\textup{\mbox{1}\hspace{-0.25em}\mbox{l}}}

\newcommand{\NCM}{N_{C/M}}

\newcommand{\KHo}{K_{\textup{H\"{o}}}}
\newcommand{\CHo}{C_{\textup{H\"{o}}}}
\newcommand{\Pic}{\mathrm{Pic}}
\newcommand{\dEuc}{d_{\mathrm{Euc}}}
\newcommand{\dUeda}{d_{\mathrm{Ueda}}}
\newcommand{\genus}{\mathrm{genus}}
\newcommand{\hflat}{h_{\mathrm{flat}}}
\newcommand{\Picntzero}{\mathrm{Pic}^0_\mathrm{nt}}
\begin{document}

\title[On holomorphic tubular neighborhoods of compact Riemann surfaces]
{On holomorphic tubular neighborhoods of compact Riemann surfaces} 

\author[S. Ogawa]{Satoshi Ogawa} 

%
\address{
Department of Mathematics, Graduate School of Science, Osaka Metropolitan University \endgraf
3-3-138, Sugimoto, Sumiyoshi-ku Osaka, 558-8585 \endgraf
Japan}
\email{sn22894n@st.omu.ac.jp}
\maketitle

\begin{abstract}
Let $C$ be a compact Riemann surface holomorphically embedded in a non-singular complex surface $M$ with the unitary flat line bundle $\NCM$. 
We give a sufficient condition for the existence of a holomorphic tubular neighborhood of $C$ in $M$.  
Our sufficient condition is described by an arithmetical condition of $\NCM$ in $\Pic^0(C)$ which can be regarded as an analogue of the Brjuno condition for irrational numbers which appears in the theory of 1-variable complex dynamics. 
\end{abstract}

\section{Introduction}\label{intro}

Let $C$ be a complex submanifold of a complex manifold $M$. 
Our interest in this paper is on  a {\em holomorphic tubular neighborhood} of $C$ in $M$: i.e.   
a neighborhood $T$ of $C$ in $M$ such that there exists a biholomorphic map $\vp$ from $T$ to a neighborhood of the zero section $C'$ of $\NCM$ whose restriction $\vp|_C\colon C \to C'$ is a biholomorphism, where $\NCM$ is the normal bundle over $C$. 
Our motivation comes from \cite{T} and \cite{KU}. 
In \cite{T}, Tsuji constructed holomorphic tubular neighborhoods of  Hopf surfaces embedded in 3-folds under a condition on normal bundles and showed the existence of a new complex structure on $S^3 \times S^3$ by using them. 
In \cite{KU}, Koike and Uehara constructed a K3 surface by using holomorphic tubular neighborhoods of elliptic curves embedded in rational surfaces  (see also \cite{L}). 
In what follows, we always suppose that $\dim C = 1$, $\dim M = 2$, and $C$ is connected. 

Historically, the first non-trivial sufficient condition for the existence of a holomorphic tubular neighborhood was obtained by Grauert \cite{G}. 
He showed that, when $\NCM$ is negative, the {\em formal principle} holds for a neighborhood of $C$ in $M$: i.e. a holomorphic tubular neighborhood of $C$ in $M$ exists if there exists a formal mapping which gives it. 
See \cite[Definition1.1, 1.2]{H} for the details of the formal principle. 

In contrast to this, it is known that the formal principle does not hold when $\NCM$ is unitary flat from \cite{A} and \cite{U}. 
In \cite{A}, Arnol'd posed the Diophantine condition for unitary flat line bundles from the view point of Siegel's linearization theorem in the theory of 1-variable complex dynamics (see \cite{S}, and also \cite[\S 2.6]{CG},  
about a generalization of Arnol'd's result, see \cite{GS} and \S\ref{GSresult} in this paper). 
He showed that, when $C$ is an elliptic curve with the unitary flat normal bundle $\NCM$, there exists a holomorphic tubular neighborhood of $C$ in $M$ if $\NCM$ satisfies the Diophantine condition. 
In \cite[\S 5.4]{U}, Ueda gave a concrete example of $C$ and $M$ with the unitary flat normal bundle in which the formal principle breaks down. 

Motivated from the improved linearization theorem on $1$-variable complex dynamics by Brjuno \cite[Theorem 6 in Chapter II]{B} (see also \cite[Th\'{e}or\`{e}me in page 6]{Y}), we will pose the {\em Brjuno condition} for unitary flat line bundles. 
Denote by $\Pic^0(C)$ the set of equivalent classes of unitary flat line bundles over $C$. 
We remark that $\Pic^0(C)$ has the holomorphically trivial line bundle $\one$.  
Let $\Picntzero(C)$ be the subset of $\Pic^0(C)$ defined by $\Picntzero(C) = \{E\in\Pic^0(C) \mid E^\nu \neq \one\ \textup{for any} \  \nu \in \Z\backslash\{0\}\}$. 
In our settings, since $\Pic^0(C)$ is biholomorphic to a complex torus, it admits the Euclidian distance $d$. 
\begin{definition}\label{Brjuno}
For $E \in \Picntzero(C)$, we say that $E$ satisfies the {\em Brjuno condition} when $E$ satisfies  
\[
\sum_{k\geq1} \frac{\log \omega_{k+1}}{2^k} < \infty, 
\]
where $\omega_{k+1}$ is defined by 
\[
\omega_{k+1} = \omega_{k+1}(E) = \max_{2 \leq \ell \leq 2^{k+1}} \frac{1}{d(\one, E^{-\ell + 1})}. 
\]
\end{definition}
In \S\ref{invariant_dist}, we will see the invariance of the Brjuno condition under the change of the choice of the Euclidian distances. 
Our main result is the following.
\begin{theorem}\label{main}
Let $C$ be a compact Riemann surface holomorphically embedded in a non-singular complex surface $M$ with the unitary flat normal bundle $\NCM$. 
Assume that $C$ has a holomorphic tubular neighborhood by a formal mapping which is tangent to the identity and preserves the splitting of $TM|_C$ in the sense of \cite{GS}. 
If $\NCM$ satisfies $\NCM\in\Picntzero(C)$ and the Brjuno condition, then $C$ actually has  a holomorphic tubular neighborhood in $M$. 
\end{theorem}
If $C$ is an elliptic curve and $\NCM \in \Picntzero(C) $, since we can see the existence of a formal mapping to give a holomorphic tubular neighborhood, we have the following corollary which is a generalization of  Arnol'd's result \cite{A}. 
\begin{corollary}\label{main_cor}
Let $C$ be an elliptic curve holomorphically embedded in a non-singular complex surface $M$  with the  unitary flat normal bundle $\NCM$. 
If $\NCM$ satisfies $\NCM\in\Picntzero(C)$ and the Brjuno condition, then $C$ has  a holomorphic tubular neighborhood in $M$. 
\end{corollary}

Theorem \ref{main} is provided by a computation of the summation which appears in \cite[Theorem 1.5]{GS}. 
In \cite{GS}, Gong and Stolovitch gave a sufficient condition for the existence of a holomorphic tubular neighborhood of $C$ in $M$ under the assumption that there exists a formal mapping which gives it. 
In their sufficient condition, a certain estimate of solutions of $\delta$-equations concerning \v{C}ech coboundary map $\delta \colon\check{C}^0(\U, \mathcal{O}_C(L)) \to \check{C}^1(\U, \mathcal{O}_C(L))$ plays an important role, where $L$ is a holomorphic line bundle over $C$ and $\U$ is a finite covering of $C$. 
In the present paper, we give an explicit estimate of the form $\|u\| \leq K \| f\|$ with $\delta u = f$ for a given $f\in\check{C}^1(\U, \mathcal{O}_C(L))$ and a solution $u \in \check{C}^0(\U, \mathcal{O}_C(L))$ of $\delta$-equation, where we used the $L^2$-norm or the $L^\infty$-norm of cochains (see \S\ref{L2KSD} and \S\ref{Linfinity}).  

\vspace{2mm}
{\bf The organization of this paper.}
In \S\ref{L2KSD}, we will review the existence theorem of a solution of $\delta$-equation with some estimates and the results in \cite{GS}.  
In \S\ref{invariant_dist}, we will see the invariance and some properties of the Brjuno condition. 
In \S\ref{Ho}, we will review the existence theorem of the $\delbar$-equation with the $L^2$-estimate for the proof of Theorem \ref{main}. 
In \S3, we will prove Theorem \ref{main}.


\vspace{2mm}
{\bf Acknowledgment.} 
The author would like to give thanks to Prof. Laurent Stolovitch and Prof. Takayuki Koike with fruitful comments. 
This work was partly supported by Osaka Central Advanced Mathematical Institute MEXT Joint Usage/Research Center on Mathematics and Theoretical Physics JPMXP0619217849), Osaka Metropolitan University,  the Research Institute for Mathematical Sciences, an International Joint Usage/Research Center located in Kyoto University, and JST, the establishment of university fellowships towards the creation of science technology innovation, Grant Number JPMJFS 2138.

\section{Preliminaries}\label{preliminaries}

\subsection{\v{C}ech coboundary equations with estimates }\label{L2KSD}
Let $C$ be a  compact Riemann surface and $g$ be a Hermitian metric on $C$. 
We fix a sufficiently fine finite open covering $\U = \{U_j\}$ of $C$. 
For $\U$ as above, denote by $\{(U_j, \vp_j)\}$ an atlas such that $\vp_j(U_j) = \D$ holds for any $j$, where $\D$ is the unit disk $\{z \in \C \mid |z| < 1\}$. 
We denote by $U_j^r$ the nested open set of $U_j$ defined by  $U_j^r = \vp_j^{-1}(\D_r)$ for a real number $r \in (0, 1]$, where $\D_r$ is the disk of radius $r$ centered at the origin. 
Then, we obtain $r_\ast \in (0, 1)$ which is sufficiently close to 1 such that, for any $r \in [r_\ast,1]$, the family $\U^r \coloneqq \{U_j^r\}$ of nested open sets  is also a sufficiently fine finite open covering of $C$. 
We call $\U^r \  (r \in(r_\ast, 1))$ a {\em nested covering} (of $\U$) of $C$. 

We review the definitions of norms on $\check{C}^j(\U^r, L) \coloneqq  \check{C}^j(\U^r, \mathcal{O}_C(L))$ for $j = 0$ or $1$, where $(L, h)$ is a Hermitian holomorphic line bundle over $C$. 
For a $0$-cochain  $u = \{(U_j^r, u_j)\} \in \check{C}^0(\U^r, L)$, we define the $L^2$-norm of  $u_j$ by
\[
\|u_j\|_{L^2, U_j^r} \coloneqq \sqrt{ \int_{U_j^r} |u_j|_{h}^2 dV_g},
\]
where $dV_g$ is the volume form determined by $g$. 
By using it, we define the norm of $u$ induced from the $L^2$-norm  by 
\[
\|u\|_{L^2, \U^r} \coloneqq \max_{j} \|u_j\|_{L^2, U_j^r}. 
\]
Similarly,  we define the norm induced from the $L^2$-norm of a $1$-cochain $f= \{(U_{jk}^r, f_{jk})\} \in \check{C}^1(\U^r, L)$ by
\[
\|f\|_{L^2, \U^r} \coloneqq \max_{j, k} \|f_{jk}\|_{L^2, U_{jk}^r} = \sqrt{ \int_{U_{jk}^r} |f_{jk}|_{h}^2 dV_g},  
\]
where we often regard $U_{jk}^r$ as $U_{jk}^r  = U_j^r \cap U_k^r$. 
We can also define the norms on $\check{C}^0(\U^r, L)$ and $\check{C}^1(\U^r, L)$ induced from the $L^\infty$-norms (see \S \ref{Linfinity}). 

\vspace{2mm}
\subsubsection{\bf Kodaira--Spencer type estimate}
In \cite{KS}, Kodaira and Spencer showed the existence of a solution of the $\delta$-equation with a certain estimate. 
In \cite{GS}, Gong and Stolovitch generalized Kodaira--Spencer's result for general complex manifolds and vector bundles. 

\begin{proposition}[{\cite[p. 499]{KS}}, {\cite[Theorem 1.1]{GS}\label{LtwoKS}}]
For any Hermitian holomorphic line bundle $(L, h)$ over a compact Riemann surface $(C, g)$ equipped with a Hermitian metric and a nested covering $\U^r$ with $r \in [r_\ast, 1]$,  there exists a constant $C_K(L)$ such that the following holds: 
for each $r \in [r_\ast, 1]$ and each \v{C}ech coboundary $f \in \check{B}^1(\U^{r}, L)$, there exists $u \in \check{C}^0(\U^{r}, L)$ such that $\delta u = f$ and 
\[
\|u\|_{L^2, \, \U^r}  \leq C_K(L)  \|f\|_{L^2,\, \U^r} 
\]
hold. 
\end{proposition}
We denote by $K(L)$ the infimum of $C_K(L)$ which satisfies the property stated in Proposition \ref{LtwoKS} for $(L, h)$. 
We remark that the constant $K(L)$ depends generally on $h$. 
In this paper, for the setting of $C, M$, and $\NCM$ in \S\ref{intro},  we study the case that $(L, h)$ belongs to the following \textbf{(a)} or \textbf{(b)} for using the result in \cite{GS}: 
 \\
\indent{\bf(a)} $L = T_C \otimes \NCM^{-\ell+1}$ and $h$ is the Hermitian fiber metric  induced from $g\otimes \hflat^{-\ell + 1}$, \\
\indent{\bf(b)} $L = \NCM^{-\ell+1}$ and  $h$ is the Hermitian fiber metric induced from $\hflat^{-\ell + 1}$. \\
Here, $T_C$ is the tangent bundle of $C$, $\hflat$ is a flat fiber metric on $\NCM$, and $\ell$ is an integer which is greater than 1. 

From the definition, $K(L)$ does not change up to the choice of a flat fiber metric $\hflat$ (for example, see \cite[Lemma 2.4]{HK}). 

 \vspace{2mm}
\subsubsection{\bf Donin type estimate}
For a nested covering $\U^r$ of $C$, we obtain an estimate related to $\delta$-equation under shrinking of  $\U^r$. 

\begin{proposition}[{\cite{D}}, {\cite[Theorem 1.1]{GS}\label{LtwoD}}]
For any Hermitian holomorphic line bundle $(L, h)$ over a compact Riemann surface $(C, g)$ equipped with a Hermitian metric and a nested covering $\U^r$ with $r \in [r_\ast, 1]$, there exists a constant $C_D(L)$ such that the following holds: 
For each $r', r''\in [r_\ast, 1]$ which satisfy  $r_\ast < r'' < r' \leq 1$ and each \v{C}ech coboundary $ f \in \check{B}^1(\U^{r'}, L)$, there exists $v \in \check{C}^0(\U^{r''}, L)$ such that $\delta v = f$ and 
\[
\|v \|_{L^2,\, \U^{r''}}\leq \frac{C_D(L)}{(r' - r'')^\tau}\,\|f\|_{L^2,\, \U^{r'}}
\]
hold, where $C_D(L)$ and $\tau$ are independent of $r'$ and $r''$. 
\end{proposition}
We denote by $D(L)$ the infimum of $C_D(L)$ which satisfies the property stated in Proposition \ref{LtwoD} for $(L, h)$.  
The constant $D(L)$ depends generally on $L, h$, and $\tau$. 
The case that we need to consider $(L, h)$ is that $(L, h)$ belongs to the case \textbf{(b)} which appeared above. 
In the case that $(L, h)$ belongs to \textbf{(b)}, it can be shown that $\tau = 2$ in Proposition \ref{correspondenceD} and Proposition \ref{UedaD}. 
From the definition, also $D(L)$ does not change up to the choice of a flat fiber metric $\hflat$ from \cite[Lemma 2.4]{HK}.

\vspace{2mm}
\subsubsection{\bf Gong--Stolovitch's result}\label{GSresult}
We fix a compact Riemann surface $(C, g)$ equipped with a Hermitian metric holomorphically embedded in a non-singular complex surface $M$ and a nested covering $\U^r \, (r \in[r_\ast, 1])$ of $C$. 
If $\NCM$ satisfies the following condition, we say that $\NCM$ satisfies the condition {\bf (GS)}: the normal bundle $\NCM$ satisfies $\NCM \in \Picntzero(C)$ and
\[
\sum_{k \geq 1} \frac{\log D_\ast(2^{k+1})}{2^k} < \infty, 
\]
where $D_\ast(2^{k+1})$ is defined by 
\[
D_\ast (2^{k+1}) = 1 + \max_{2 \leq \ell \leq 2^{k+1}} \{ (1+ c\,K(T_C\otimes \NCM^{-\ell + 1})) \cdot D(\NCM^{-\ell + 1})\}. 
\]
Here $c$ is a positive constant determined by the initial settings $\{C, M, \NCM\}$ (see \cite[equation 5.25]{GS}). 
The condition \textbf{(GS)} makes sense for compact complex submanifolds of any dimension which have a higher codimension.  
Gong and Stolovitch gave this condition for the existence of a holomorphic tubular neighborhood in general. 
In \cite{GS}, they discussed the existence of a holomorphic tubular neighborhood by using $D_\ast(2^{k+1})$ which is consisted by $K$ and $D$ related to the $L^\infty$-norm. 
From \cite[\S A.2]{GS} and \cite[Chapter VI]{GR}, norms of \v{C}ech cochains induced the $L^2$-norm and the $L^\infty$-norm are equivalent up to scale. 
Therefore, we may investigate the condition \textbf{(GS)} by using the $L^2$-norm.
In \S\ref{Linfinity}, we will study the existence of a solution of $\delta$-equation with a certain estimate by the norms induced from $L^\infty$-norms.
 
In the following, we say that {\em $TM|_C$ splits} if the short exact sequence $0 \to T_C \to TM|_C \to \NCM \to 0$ splits (i.e. $TM|_C = T_C \oplus \NCM$). 

\begin{theorem}[{\cite[Theorem 1.5]{GS}} for our case]
Let $C$ be a compact Riemann surface holomorphically embedded in a non-singular complex surface $M$ with the unitary flat normal bundle $\NCM$ and $\U^r$ be a nested finite covering of $C$. 
Assume that either $TM|_C$ splits and $H^1(C, TM|_C \otimes \NCM^{-\ell}) = 0$ for all $\ell > 1$, or $C$ has a holomorphic tubular neighborhood in $M$ by a formal holomorphic mapping which is tangent to the identity and preserves the splitting of $TM|_C$ in the sense of \cite{GS}.
If $\NCM$ satisfies condition {\rm\bf{(GS)}}, then $C$ actually has a holomorphic tubular neighborhood in $M$. 
\end{theorem}
\noindent{\bf Remark for Corollary \ref{main_cor} }
In our case, we can see that $TM|_C$ splits and $H^1(C, TM|_C \otimes \NCM^{-\ell}) = 0$ holds for all $\ell > 1$ if and only if the genus of $C$ is  $0$ or $1$ under the assumption that $\NCM \in \Picntzero(C)$.  	

Especially, when the genus of $C$ is equal to $0$, the normal bundle $\NCM$ is the holomorphically trivial line bundle. 
Then, it is obvious that the condition \textbf{(GS)} holds. 
It is consistent to the result in \cite{Sa} about the existence of a holomorphic tubular neighborhood of $\mathbb{P}^1$ . 
Hence, we always suppose that the genus of $C$ is greater than or equal to $1$ in this paper. 
\subsection{On the Brjuno condition }\label{invariant_dist}
In this subsection, we will see fundamental properties of  the Brjuno condition for unitary flat line bundles. 
Let $C$ be a compact Riemann surface and $\Pic^0(C)$ the space of equivalent classes of unitary flat line bundles over $C$. 
Then, $\Pic^0(C)$ is the connected component of the Picard variety of $C$ which has the holomorphically trivial line bundle $\one$.  
We induce an {\em invariant distance} on $\Pic^0(C)$ in the following sense: for any $E_1, E_2, E_3 \in \Pic^0(C)$, $d(E_1, E_2) = d(E_1^{-1}, E_2^{-1}) = d(E_1 \otimes E_3, E_2 \otimes E_3)$ holds. 

Since $\Pic^0(C)$ is homeomorphic to a complex torus in our settings, it admits an invariant distance $\dEuc$ induced from an Euclidian distance on the universal covering space of $\Pic^0(C)$. 

\vspace{2mm}
\subsubsection{\bf The invariance of Brjuno condition}
Firstly, we see that the Brjuno condition is invariant under the choice of an Euclidian distance. 
We say that $E\in\Picntzero(C)$ {\em  is Brjuno} when $E$ satisfies the Brjuno condition (Definition \ref{Brjuno}) for instance. 
\begin{proposition}\label{Equiv_Euc}
Let $\dEuc$ and $\dEuc'$ be the distances on $\Pic^0(C)$ induced from two Euclidian distances on the universal covering space of $\Pic^0(C)$ and $E\in\Picntzero(C)$. 
Then, that $E$ is Brjuno in the sense of $\dEuc$ is equivalent to that $E$ is Brjuno in the sense of $\dEuc'$. 
\end{proposition}
\begin{proof}
The distance $\dEuc$ is Lipschitz equivalent to $\dEuc'$: i.e. there exists a positive constant $\Lambda$ such that, for any $E_1, E_2 \in \Pic^0(C)$, the following relation holds:
\[
\frac{1}{\Lambda} \cdot  \dEuc(E_1, E_2) < \dEuc'(E_1, E_2) < \Lambda \cdot \dEuc(E_1, E_2). 
\]
Then, we have
\[
\Lambda \cdot \max_{2 \leq \ell \leq 2^{k+1}} \frac{1}{\dEuc(\one, E^{-\ell + 1})} \leq \max_{2 \leq \ell \leq 2^{k+1}} \frac{1}{\dEuc'(\one, E^{-\ell + 1})} \leq \Lambda^{-1} \cdot \max_{2 \leq \ell \leq 2^{k+1}} \frac{1}{\dEuc(\one, E^{-\ell + 1})} 
\]
for any $k \geq 1$. 
Taking the logarithm and the summation over $k\geq 1$, we complete the proof. 
\end{proof}

We can introduce an other invariant distance on $\Pic^0(C)$ called the {\em Ueda's distance} \cite[\S 4.1]{U}. 
Let $\dUeda$ be the distance on $\Pic^0(C)$ defined by
\[
\dUeda(\one, E) \coloneqq \inf\{ \max_{j, k} |1-t_{jk}| \mid E = [\{(U_{jk}, t_{jk}) \}] \in \check{H}^1(\U, \mathrm{U}(1))\}, 
\]
where $\U = \{U_j\}$ is a sufficiently fine finite open covering of $C$. 
\begin{proposition}[{\cite[Proposition A.3]{KU}\label{equiv}}]
Let $\dEuc$ be a one of the invariant distances on $\Pic^0(C)$ induced from an Euclidian distance on the universal covering space of $\Pic^0(C)$, and $\dUeda$ the invariant distance in the sense of Ueda. 
Then, $\dEuc$ is Lipschitz equivalent to $\dUeda$ as a distance. 
\end{proposition}
From this proposition, we can obtain the following. 

\begin{proposition}
Let $E\in\Picntzero(C)$. 
It is equivalent that $E$ is Brjuno in the sense of $\dEuc$ to that $E$ is Brjuno in the sense of $\dUeda$. 
\end{proposition}
\begin{proof}
It can be proven similarly to the proof of Proposition \ref{Equiv_Euc}. 
\end{proof}
Finaly, we see that the Brjuno condition determines a larger class of $\Picntzero(C)$ than the Diophantine condition. 
We say that $E \in \Picntzero(C)$ satisfies the {\em Diophantine condition} when there exist positive numbers $c$ and $\tau$ such that $d(\one, E^n) \geq cn^{-\tau}$ holds for any positive integer $n$. 
\begin{proposition}\label{DiophBrjuno}
If $E \in \Picntzero(C)$ satisfies the Diophantine condition, then $E$ satisfies the Brjuno condition. 
\end{proposition}
\begin{proof}
Let $c$ and $\tau$ be positive numbers such that $d(\one, E^n) \geq cn^{-\tau}$ holds for any positive integer $n$. 
Then, one can see
\begin{align*}
\omega_{k+1}(E) = \max_{2 \leq \ell \leq 2^{k+1}} \frac{1}{d(\one , E^{-\ell +1 })} 
&= \max_{2 \leq \ell \leq 2^{k+1}}  \frac{1}{d(\one , E^{\ell -1 })}\\
&\leq \max_{2 \leq \ell \leq 2^{k+1}} \frac{(\ell-1)^\tau}{c} < \frac{2^{\tau(k+1)}}{c}, 
\end{align*}
for each $k\geq1$. Therefore, one has 
\[
\sum_{k \geq 1} \frac{\log \omega_{k+1}(E)}{2^k} < \sum_{k\geq1} \frac{\tau(k+1)\log2-\log c}{2^k}  < \infty.
\]
\end{proof}
\vspace{2mm}
\subsubsection{\bf Example}
Here, we use the Ueda's distance $\dUeda$ to see the Brjuno condition of a certain flat line bundle. 
For a compact Riemann surface $C$ whose genus is greater than or equal to $1$, let $\{U_j\}$ be a finite covering of $C$. 
In what follows, we suppose that the flat line bundle $E$ is induced the data $\{(U_{jk}, t_{jk})\} \in \check{C}^1(\{U_j\}, \mathrm{U}(1))$, where $\mathrm{U}(1) = \{ z \in \C \mid |z| = 1\}$.  
We also assume that unitary constants $t_{jk}$ satisfy the following relation:
\[
t_{jk} = 
	\begin{cases}
	e^{\twopii\theta} & (jk = 01) \\
	e^{- \twopii\theta} & (jk = 10) \\
	1 & \textup{otherwise}, 
	\end{cases}
\]
where $\theta$ is an irrational number. 
Then, the Ueda's distance between $\one$ and $E^n$ can be estimated as the following by using a sufficiently small positive number $\gamma$ for each $n$: 
\[
\gamma|1 - e^{\twopii n\theta}| \leq \dUeda(\one, E^n) \leq |1 - e^{\twopii n\theta}|. 
\]
In more details, see \cite[\S A.3]{KU}. From an easy computation, 
\[
 4\mathrm{dist}(n\theta, \Z) \leq  |1 - e^{\twopii n \theta}| \leq 2\pi \mathrm{dist}(n\theta, \Z) 
\]
holds, where $\mathrm{dist}(\alpha, \Z) = \inf_{N \in \Z}|\alpha - N|$. 
Then, we have 
\[
\log \frac{1}{2\pi} \leq \log \left(\max_{2 \leq \ell \leq 2^{k+1}}\frac{1}{\dUeda(\one, E^{-\ell+1})} \right)- \log \left(\max_{2 \leq \ell \leq 2^{k+1}}\frac{1}{\mathrm{dist}((-\ell+1)\theta, \Z)} \right)\leq \log \frac{1}{4\gamma} 
\]
holds for any $k\geq1$. 
Therefore, we can see that $E\in\Picntzero(C)$ satisfies the Brjuno condition if and only if 
\[
\sum_{k\geq 1} \frac{1}{2^k} \log \left( \max_{2 \leq \ell \leq 2^{k+1}} \frac{1}{\mathrm{dist}((-\ell + 1)\theta, \Z)} \right) < \infty
\]
holds. 
The last condition is the Brjuno condition for irrational numbers (see \cite[Condition $\omega$]{B}). 
For more details of the Brjuno condition of irrational numbers, for example, see \cite{GL} and \cite{MMY}. 
\if
We can see $\Pic(C)$ is isomorphic to $\C/\langle 1, w \rangle$ for $w\in\C$ which satisfies $\mathrm{Im}(w) > 0$. 
We denote that $\sigma$ an isomorphism from $\Pic^0(C)$ to $\C/\langle 1, w \rangle$ and one of invariant distances $d$ is induced from Euclidian distance on $\C$. 
Then, we remark $\sigma(\one) = 0 \in \C/\langle 1, w \rangle$. 
We take a flat line bundle $E$ which satisfies $\sigma(E) = [\theta] \in \C/\langle 1, w \rangle$ for $\theta \in (0,1) \cap \R\backslash \mathbb{Q}$. 
It is easily checked that $E^n \neq \one$ for any $n \neq 0$ and $d(\one, E) = \min \{ \theta, 1-\theta \}$. 
We can obtain 
\[
d(\one, E^n) \leq \min_{p\in\mathbb{Z}} |n\theta-p|
\]
By using best approximation property of continued fraction {\color{red}(参考文献！)}, 
for $\nu$-th approximation continued fraction $p_\nu / q_\nu$, we can obtain
\[
d(\one, E^{q_\nu}) = (-1)^\nu (q_\nu \theta - p_\nu) 
\]
\fi
\subsection{On H\"{o}rmander type estimate}\label{Ho}
In this subsection, we review the existence of solutions to the $\delbar$-equations with $L^2$-estimate. 
%
%
Let $(C, g)$ be a compact Riemann surface equipped with a Hermitian metric and $(L, h)$ a holomorphic Hermitian line bundle over $C$. 
The following statement originating in \cite{Ho} is known as {\em H\"{o}rmander's estimate} : there exists a positive constant $\CHo(L)$ such that, for any $L$-valued $\delbar$-closed $\mathcal{C}^\infty$-class $(0, 1)$-form $v$ whose cohomology class $[v]\in H^{0, 1}(C, L)$ is trivial, then there exist a  $\mathcal{C}^\infty$-class $L$-valued global section $u$ such that $\delbar u = v$ and
\[
\|u\|_{L^2} \leq \CHo(L) \|v\|_{L^2}
\]
hold, where the norms $\|\cdot\|_{L^2}$ are defined by
\[
\|u\|_{L^2} \coloneqq  \sqrt{\int_C |u|_{h}^2 dV_{g}} \ \   {\textup{and}} \ \ 
 \|v\|_{L^2} \coloneqq  \sqrt{\int_C |v|_{g, h}^2 dV_{g}}. 
\]
We denote by $\KHo(E)$ the infimum of $\CHo(L)$ which satisfies the condition as above. 
\begin{proposition}[{\cite[Theorem 1.1]{HK} for a 1-dimensional manifold}]\label{HK}
Let $(C, g)$ be a  compact Riemann surface equipped with a Hermitian metric. 
There exists a positive constant $K_0$ such that, for any $E\in \Pic^0(C)\setminus\{\one\}$ and any $\delbar$-closed $(0, 1)$-form $v$ with values in $E$ which satisfies that the Dolbeault cohomology class $[v] \in H^{0, 1}(C, E)$ is trivial, there exists a unique smooth global section $u$ of $E$ such that $\delbar u = v$ and
\[
\|u\|_{L^2}  \leq \frac{K_0}{d(\one, E)}\ \|v\|_{L^2}
\]
hold for a flat fiber metric $\hflat$ on $E$. 
\end{proposition}
Proposition \ref{HK} implies that $\KHo(E) \leq K_0/d(\one, E)$ holds for any $E\in \Pic^0(C)\setminus\{\one\}$. 
In \cite{HK}, Hashimoto and Koike showed Proposition \ref{HK} for a compact K\"{a}hler manifold of any dimension. 

\section{Proof of Theorem \ref{main}} 
What we want to estimate are $K(T_C \otimes \NCM^{-\ell +1})$, $K(\NCM^{-\ell +1})$, and $D(\NCM^{-\ell + 1})$ which are defined in \S \ref{L2KSD} under the assumption that $\NCM \in \Picntzero(C)$ to compute $D_\ast(2^{k+1})$ in {GSresult}. 

\subsection{\v{C}ech--Dolbeault correspondence}
We begin with the following proposition, whose proof is inspired by  \cite[Lemma 2.8]{HK}. 
\begin{proposition}\label{correspondenceKS}
Let $(L, h)$ be a Hermitian line bundle over a compact Riemann surface $(C, g)$ equipped with a Hermitian metric and $\U^r$ a nested covering of $C$ with radius $r \in [r_\ast, 1]$. 
Then, there exist positive constants $\alpha$ and $\beta$ such that, 
for any $f \in\check{B}^1(\U^r,\, L)$, there exists $u\in\check{C}^0(\U^r, L)$  which satisfies $\delta u = f$ and  $\|u\|_{L^2, \,\U^r} \leq (\alpha + \beta \KHo(L) ) \|f\|_{L^2, \,\U^r}$.
Here, the constants $\alpha$ and $\beta$ are independent of $(L, h)$ and $r$. 
\end{proposition}
These norms $\|\cdot\|_{L^2,\, \U^r}$ are defined in \S \ref{L2KSD} and the constant $\KHo(L)$ is defined in \S \ref{Ho}. 
Denote by $\{(U_{jk}^r, f_{jk})\}$ a 1-coboundary $f\in\check{B}^1(\mathcal{U}^{r}, L)$. 
Fix a partition of unity $\{(U_j^{r_\ast}, \rho_j)\}$ for the minimum radius $r_\ast$ and suppose that $\|f_{jk}\|_{L^2, U_{jk}^{r}} < \infty$. 
We define $F_j$ by $F_j = \sum_{\ell \neq j} \rho_\ell f_{j\ell}$ on each $U_j^r$. 
\begin{lemma}\label{lemmaA}
There exists a positive number $C_1$ such that 
\[
\max_j \|\delbar F_{j}\|_{L^2, U_j^r} \leq C_1 \cdot \|f\|_{L^2, \, \U^r}. 
\]  
\end{lemma}
\begin{proof}
Since each $f_{jk}$ is holomorphic, we can see
\[
\|\delbar F_{j}\|_{L^2, U_j^r} \leq \sum_{k \neq j} |\delbar\rho_j|_{g} |f_{jk}|_h \leq C_1  \mathcal{N} \cdot \|f\|_{L^2, \, \U^r}
\]
with $C_1 \coloneqq \max_j \sup_{U_j^1}|\delbar\rho_j|_{g}$. 
We remark that $C_1$ depends only on $\U^r$ and a partition of unity. 
Taking the maximum over all $j$,  we complete the proof. 
\end{proof}
\begin{lemma}\label{lemmaB}
Denote by $\mathcal{N}$ the number of a finite covering $\mathcal{U}^1$ of $C$. Then, the following holds: 
\[
\max_j \|F_{j}\|_{L^2, U_j^r} \leq \mathcal{N} \cdot \|f\|_{L^2, \, \U^r}.
\] 
\end{lemma}
\begin{proof}
Note that $\mathcal{N}$ is invariant for any $r \in [r_\ast, 1]$. 
From the definition
, one can see
\[
\| F_{j}\|_{L^2, U_j^r} \leq \sum_{k \neq j} \| f_{jk}\|_{L^2, U_{jk}^r} \leq  \mathcal{N} \cdot \|f\|_{L^2, \, \U^r}.
\]
Taking the maximum over all $j$, we complete the proof. 
\end{proof}

\begin{proof}[\indent{\sc Proof of proposition \ref{correspondenceKS}}. ]
From the definition, $\delbar F_j$ determines $L$-valued $(0, 1)$-form $F$ on $C$ such that $[F] = 0$ . 
By Lemma \ref{lemmaA}, we obtain
\begin{align*}
\sqrt{\int_C |F|_{g, h}^2 dV_g} \leq \sqrt{\sum_{j} \|F|_{U_j^{r}} \|^2_{L^2, U_j^{r}} }
&= \sqrt{\sum_{j} \|\delbar F_j \|^2_{L^2, U_j^{r}} }\\
&\leq \sqrt{ \mathcal{N} \cdot \max_j  \|\delbar F_j \|^2_{L^2, U_j^{r}}}\leq C_1 \mathcal{N}^{3/2} \cdot \|f\|_{L^2, \, \U^r}. 
\end{align*}
There exists a $L$-valued global section $\eta$ of $L$ such  that $\delbar\eta = F$ and
\[
\sqrt{\int_C |\eta|_h^2 dV_g} \leq \KHo(L) \sqrt{\int_C |F|_{g, h}^2 dV_g}. 
\]
We define $u_j$ by $u_j = F_j - \eta$ on each $U_j^r$. 
Then, $\delta u= \delta \{ (U_j^r, u_j)\}= \{(U_{jk}^r, F_k - F_j )\} = f$ and 
\begin{align*}
\|u\|_{L^2, \, \U^r} = \max_j \,\sqrt{\int_{U_j^{r}} |u_j|_h^2 dV_g}&\leq \max_j \|F_j \|_{L^2, U_j^{r}} + \sqrt{\int_C |\eta|_h^2 dV_g}\\
&\leq \left( \mathcal{N} + C_1  \mathcal{N}^{3/2} \KHo(L)\right) \cdot \|f\|_{L^2, \U^{r}}
\end{align*}
hold by using Lemma \ref{lemmaB}. 
\end{proof}
If $L$ is a unitary flat line bundle over a compact Riemann surface, we can obtain ``Ueda's lemma type statement'' (we will introduce Ueda's lemma in \S \ref{Linfinity} as Proposition \ref{Ueda}). 
\begin{corollary}\label{CorKS}
For any unitary flat line bundle $E\in \Pic^0(C)\setminus\{\one\}$, there exists a positive constant $K_1$ which is independent of $E$ such that 
\[
K(E) \leq \frac{K_1}{d(\one, E)} 
\]
holds, where d is the distance on $\Pic^0(C)$ induced from the Euclidian distance. 
\end{corollary}
\begin{proof}
From Proposition \ref{correspondenceKS}, $K(E) \leq \alpha+ \beta \KHo(E)$ holds by using constants $\alpha$ and $\beta$ which are independent of $E$. 
There exists a positive constant $K_0$ which is independent of $E$ such that $\KHo(E) \leq K_0/d(\one, E)$ from Proposition \ref{HK}. 
Then, we have
\[
K(E) \leq \alpha +  \frac{\beta K_0}{d(\one, E)} \leq \frac{\alpha\Delta +  \beta K_0}{d(\one, E)}, 
\]
where $\Delta$ is the diameter of $\Pic^0(C)$ measured by the distance $d$. 
Then, $K_1 \coloneqq \alpha\Delta +  \beta K_0$ is independent of $E$. 
\end{proof}
Next, we see that we obtain the Donin type estimate from the Kodaira--Spencer type estimate when a holomorphic line bundle over $C$ is unitary flat. 
\begin{proposition}\label{correspondenceD}
Let $(C, g)$ be a compact Riemann surface equipped with a Hermitian metric, $(E, \hflat)$ a unitary flat line bundle over $C$, and $\U^r = \{U_j^r\} (r \in [r_\ast, 1])$ a nested finite covering of $C$. 
There exists a positive constant $D$ which is independent of $E$ and all $r \in[r_\ast, 1]$ and satisfies the following:
for any $f\in\check{B}^1(\mathcal{U}^{r_1}, E)$, there exists $u \in \check{C}^0(\mathcal{U}^{r_2}, E)$ which satisfies $\delta u  = f$, and  
\[
\|u\|_{L^2, \, \U^{r_2}}  \leq \frac{D \cdot K(E)}{(r_1- r_2)^2}\cdot \|f\|_{L^2, \, \U^{r_1}}. 
\]
Here, $r_1$ and $r_2$ are any numbers satisfying $r_\ast < r_2 < r_1 \leq  1$. 
\end{proposition}
\begin{proof}
By Proposition \ref{LtwoKS}, for a given 1-coboundary $f\in\check{B}^1(\mathcal{U}^{r_1}, E)$, there exists $u = \{(U_j^{r_1}, u_j)\} \in \check{C}^0(\mathcal{U}^{r_1}, E)$ which satisfies $\delta u  = f$, and  
\[
\|u\|_{L^2, \, \U^{r_1}}  \leq  K(E) \cdot \|f\|_{L^2, \, \U^{r_1}}. 
\]
For instance, we write  $u = \{(U_j^{r_2}, u_j)\}$ instead of $\{(U_j^{r_2}, u_j|_{U_j^{r_2}})\}$ and we note that $\delta \{(U_j^{r_2}, u_j)\} = f$ holds. 
Then, $\|u\|_{L^2, \, \U^{r_2}}$ is what we want to compare to  $\|u\|_{L^2, \, \U^{r_1}}$ for any $r_\ast < r_2 < r_1 \leq 1$. 
We can take a positive-valued $\mathcal{C}^\infty$-class function $a_j$ on $U_j^1$ which satisfies 
\[
dV_g|_{U_j^1} = a_j d\lambda_j \coloneqq \frac{\sqrt{-1}}{2} \,a_j \,dz_j \wedge d\overline{z_j},
\]
for each $U_j^1$. 
Let $A^-$ and $A^+$ be positive constants defined by 
\[
A^- = \sqrt{\min_j \inf_{U_j^1} a_j}, \ \textup{and} \ A^+ = \sqrt{\max_j \sup_{U_j^1} a_j}. 
\]
We remark that the above $A^-$ and $A^+$ are independent of $r_1$ and $r_2$. 
Then, we can obtain the following two inequalities: 
\begin{gather}
\|u\|_{L^2, \, \U^{r_2}} \leq A^+ \max_j \sqrt{\int_{U_j^{r_2}} |u_j|^2 d\lambda_j}, \label{plus}
\\
 A^- \max_j \sqrt{\int_{U_j^{r_1}} |u_j|^2 d\lambda_j} \leq \|u\|_{L^2, \, \U^{r_1}}. \label{minus}
\end{gather}
By using the mean value inequality, for each $\xi \in U_j^{r_2}$, one can estimate as follows: 
\begin{align*}
|u_j(\xi)|^2 &\leq \frac{1}{\mathrm{Vol}(\D(\xi, r_1 - r_2))^2}\int_{\D(\xi, r_1- r_2)}|u_j|^2d\lambda_j \\
&\leq  \frac{1}{(\pi(r_1 - r_2)^2)^2} \max_j \,{\int_{U_j^{r_1}} |u_j|^2 d\lambda_j},
\end{align*}
where $\D(\xi, r_1 - r_2)$ is the disk of radius $r_1 - r_2 >0$ centered at $\xi$  . 
Taking the integral on $U_j^{r_2}$, one can have
\begin{align*}
\int_{U_j^{r_2}} |u_j|^2 d\lambda_j &\leq  \frac{1}{(\pi(r_1 - r_2)^2)^2} \max_j \,{\int_{U_j^{r_1}} |u_j|^2 d\lambda_j} \int_{U_j^{r_2}}  d\lambda_j \\
&\leq \frac{1}{(\pi(r_1 - r_2)^2)^2} \max_j \,{\int_{U_j^{r_1}} |u_j|^2 d\lambda_j} \int_{U_j^1}  d\lambda_j. 
\end{align*}
By taking the square root and the maximum over all  $j$, one can see
\begin{equation}\label{V}
\max_j \sqrt{\int_{U_j^{r_2}} |u_j|^2 d\lambda_j} \leq \frac{V}{\pi(r_1 - r_2)^2} \max_j \sqrt{ \int_{U_j^{r_1}} |u_j|^2 d\lambda_j}
\end{equation}
by letting $V = \sqrt{\max_j \mathrm{Vol}(U_j^1)}$. 
By combining  inequalities (\ref{plus}), (\ref{minus}), and (\ref{V}), one can obtain
\[
\|u \|_{L^2, \, \U^{r_2}} \leq \frac{D}{(r_1 - r_2)^2} \|u \|_{L^2, \, \U^{r_1}}  \leq \frac{D \cdot K(E)}{(r_1- r_2)^2}\cdot \|f\|_{L^2, \, \U^{r_1}}
\]
by letting
\[
D = \frac{VA^+}{\pi A^-}. 
\]
We remark that the positive constant $D$ is independent of $E$, $r_1$, and $r_2$.

\end{proof}
\begin{corollary}\label{CorD}
For any unitary flat line bundle $E\in \Pic^0(C)\setminus\{\one\}$, there exists a positive constant $D_1$ which is independent of $E$ such that 
\[
D(E) \leq \frac{D_1}{d(\one, E)}.  
\]
holds, where d is the distance on $\Pic^0(C)$ induced from the Euclidian distance. 
\end{corollary}
\begin{proof}
From Corollary \ref{CorKS} and Proposition \ref{correspondenceD}, one can obtain
\[
D(E) \leq D \cdot\frac{K_1}{d(\one, E)}.
\]
The positive constant $D_1 \coloneqq DK_1$ is independent of $E$.  

\end{proof}
\subsection{An estimate of $K(T_C\otimes E)$ by using perturbed $\delbar$-operators}
In this subsection, we always assume that $C$ is a compact Riemann surface whose genus is greater than 1 and $E$ is a unitary flat line bundle over $C$. 
Let $g$ be a Hermitian metric and $\hflat$ a flat fiber metric on $E$. 
What we want to investigate is $K(T_C \otimes E)$, where $T_C$ is the tangent bundle of $C$. 
We always suppose that a Hermitian fiber metric on $T_C \otimes E$ is induced from $g\otimes \hflat$. 
The goal of this subsection is proving the following proposition. 
\begin{proposition}\label{bdd}
Let $C$ be a compact Riemann surface whose genus is grater than 1. 
For $E\in \Pic^0(C)$, $K(T_C \otimes E)$ has the maximum. 
\end{proposition}

Let $\kappa$ be a function $\kappa\colon \Pic^0(C) \to \mathbb{R}$ defined by
\[
\kappa(E) = \sup \left\{  \frac{\| u \|_{L^2}}{\| \delbar u \|_{L^2}}\  \middle| \ u \in \mathcal{A}^{0, 0}(C, T_C\otimes E)\setminus\{ 0\}\right\}. 
\]
Here, $\mathcal{A}^{p, q}(C, L)$ is the space of  $\mathcal{C}^\infty$-class $L$-valued $(p, q)$-forms on $C$. 
From the assumption as above, one has $H^0(C, T_C \otimes E) = 0$ because $T_C \otimes E$ is a negative line bundle over $C$. 
Then, since $H^0(C, T_C \otimes E) = 0$ implies the uniqueness of the solution of $\delbar$-equation on $\mathcal{A}^{0, 0}(C, T_C\otimes E)$, one can obtain $\kappa(E) = \KHo(T_C \otimes E)$. 
Therefore, $\kappa$ is  well-defined as a real-valued function on $\Pic^0(C)$ in settings of this subsection. 
\begin{proposition}\label{bounded}
Let $C$ be a compact Riemann surface whose genus is greater than $1$. 
Then, the function $\kappa\colon \Pic^0(C) \to \mathbb{R}$ defined as above is an upper semi-continuous function. 
In particular, $\kappa$ has the maximum. 
\end{proposition}
\begin{proof}
Because $\Pic^0(C)$ is compact, if $\kappa$ is an upper semi-continuous function, then $\kappa$ has the maximum. 

In this paper, we introduce a proof by using the {\em perturbed $\delbar$-operator} which is provided in  \cite{HK}. 
In more generally, we can see the continuity of $\kappa$ (see \cite[\S7]{K}). 

For $E\in\Pic^0(C)$, let $\{E_n\}$ be a sequence which converges to $E$. 
Then, for each $n$, there exists a nowhere vanishing $\mathcal{C}^\infty$-class section $\sigma_n$ from $C$ to $E^{-1} \otimes E_n$ which satisfies the following conditions: 
\begin{itemize}
\item For each $p\in C$, $|\sigma_n(p)|_{\hflat} = 1$ holds,  
\vspace{1mm}
\item $\displaystyle{\lim_{n \to \infty}\left\|\frac{\delbar \sigma_n}{\sigma_n} \right\|_{L^2} = 0}$, 
\end{itemize}
where, $\hflat$ is a flat fiber metric on $E^{-1} \otimes E_n$. 
About the construction of $\sigma_n$, see \cite[\S 2.6]{HK}. 
By using each $\sigma_n$, we can define an operator $D_n\colon \mathcal{A}^{0, 0}(C, T_C \otimes E)  \to  \mathcal{A}^{0, 1}(C, T_C \otimes E)$ which is called the {\em perturbed $\delbar$-operator} by
\[
D_n u' \coloneqq \frac{1}{\sigma_n} \otimes  \delbar ( u' \otimes \sigma_n )  = \delbar u' + \frac{\delbar \sigma_n}{\sigma_n} \wedge u'. 
\]
We can compute $\kappa(E_n)$ by using $D_n$ as follows:   
\begin{align}
\kappa(E_n) \nonumber&= \sup \left\{  \frac{\| u \|_{L^2}}{\| \delbar u \|_{L^2}}\  \middle|\  u \in \mathcal{A}^{0, 0}(C, T_C\otimes E_n) \backslash\{ 0\}\right\} \\
\nonumber&= \sup \left\{  \frac{\| \sigma_n u' \|_{L^2}}{\| \delbar (\sigma_n u') \|_{L^2}}\  \middle|\  u' \in \mathcal{A}^{0, 0}(C, T_C\otimes E) \backslash\{ 0\}\right\} \\
&= \sup \left\{  \frac{\| u' \|_{L^2}}{\| D_n u' \|_{L^2}}\  \middle|\  u' \in \mathcal{A}^{0, 0}(C, T_C\otimes E) \backslash\{ 0\}\right\}. \label{kappa}
\end{align}
We will show that $\kappa$ is an upper semi-continuous function on $\Pic^0(C)$. 
We assume that $\kappa$ is not an upper semi-continuous function and we will lead to a contradiction. 
We suppose that there exists a positive constant $\ve$ such that there exists a sequence $\{E_n\}$ which converges to $E$ and 
\[
\limsup_{ n \to \infty } \kappa(E_n) > \kappa(E) + \ve
\]
holds. 
Hence, we can obtain a subsequence $\{E_{n_j}\}$ of $\{E_n\}$ which satisfies
\[
\kappa(E_{n_j}) > \kappa(E) + \frac{\ve}{2}
\]
for each $j$. 
Therefore, from the equation (\ref{kappa}), we can obtain $u_j \in \mathcal{A}^{0, 0}(C, T_C\otimes E)\backslash \{0\}$ such that
\[
\frac{\| u_j \|_{L^2}}{\| D_{n_j} u_j \|_{L^2}} \geq \kappa(E) + \frac{\ve}{2}
\]
holds for each $j$. 
Since we may assume $\|u_j\|_{L^2} = 1$, we can compute as following: 
\begin{align*}
1 &\geq \left( \kappa(E) + \frac{\ve}{2} \right)  \| D_{n_j} u_j\|_{L^2} \\
&= \left( \kappa(E) + \frac{\ve}{2} \right) \left\| \delbar  u_j + \frac{\delbar \sigma_{n_j}}{\sigma_{n_j}} \wedge u_j \right\|_{L^2} \\
&\geq \left( \kappa(E) + \frac{\ve}{2} \right) \left| \|\delbar u_j\|_{L^2} - \left\| \frac{\delbar\sigma_{n_j}}{\sigma_{n_j}} \right\|_{L^2} \cdot \| u_j\|_{L^2} \right| \\
&\geq \left( \kappa(E) + \frac{\ve}{2} \right)  \left| \frac{1}{\kappa(E)} - \left\| \frac{\delbar\sigma_{n_j}}{\sigma_{n_j}} \right\|_{L^2}\right| \\
&= \left( 1 + \frac{\ve}{2\kappa(E)} \right)  \left| 1 - \kappa(E) \left\| \frac{\delbar\sigma_{n_j}}{\sigma_{n_j}} \right\|_{L^2}\right|. 
\end{align*}
Therefore, we can lead to a contradiction in $j$ when $\left\| \frac{\delbar\sigma_{n_j}}{\sigma_{n_j}} \right\|_{L^2}$ is sufficiently small. 
\end{proof}
This approach works well when we replace $T_C$ with a negative holomorphic line bundle over $C$.

\begin{proof}[\sc Proof of proposition \ref{bdd}]
From Proposition \ref{CorKS} and Proposition \ref{bounded}, one has
\[
 K(T_C \otimes E) \leq \alpha + \beta \cdot \left(\max_{E' \in \Pic^0(C)}\KHo(T_C \otimes E') \right)< \infty
\]
with positive constants $\alpha$ and $\beta$ which are independent of  $E$. 
\end{proof}

\subsection{Proof of Theorem \ref{main}}
For proving Theorem \ref{main}, it is enough to show the following proposition. 
\begin{proposition}
Let $C$ be a compact Riemann surface holomorphically embedded in a non-singular complex surface $M$. 
Denote by $\NCM$ the normal bundle of $C$ which satisfies $\NCM \in \Picntzero(C)$. 
If $\NCM$ satisfies the Brjuno condition in the sense of Definition \ref{Brjuno}, then
\[
\sum_{k \geq 1} \frac{\log D_\ast(2^{k+1})}{2^k} < \infty, 
\]
where $D_\ast(2^{k+1})$ is the number defined in \S\ref{GSresult}.
\end{proposition}
\begin{proof}
Denote by $\genus(C)$ the genus of $C$. 
When $\genus(C) = 1$, then the tangent bundle $T_C$ is holomorphically trivial. 
From Corollary \ref{CorKS} and Corollary \ref{CorD},  there exist positive constants $K_1$ and $D_1$ such that
$K( \NCM^{-\ell + 1}) \leq  {K_1}/{d(\one, \NCM^{-\ell + 1})}$ and $D( \NCM^{-\ell + 1}) \leq  {D_1}/{d(\one, \NCM^{-\ell + 1})}$ hold, 
where, $K_1$ and $D_1$ are independent of $\NCM$. 
Letting 
\[
\omega_{k+1} = \max_{2 \leq  \ell \leq 2^{k+1}} \frac{1}{d(\one, \NCM^{-\ell + 1})}, 
\]
then $D_\ast(2^{k+1})$ is bounded from above by $1 + (1 + cK\omega_{k+1})D\omega_{k+1}$. 
Therefore, there exist constants $c_1$ and $c_2$ such that $\log D_\ast(2^{k+1}) \leq c_1 + c_2\log \omega_{k+1}$ holds. 

When  $\genus(C) > 1$, from Corollary \ref{CorD} and Proposition \ref{bdd}, there exist positive constants $\widehat{K}$ and $D_1'$
such that $K( T_C \otimes \NCM^{-\ell + 1}) \leq  \widehat{K}$ and $D( \NCM^{-\ell + 1}) \leq  {D_1'}/{d(\one, \NCM^{-\ell + 1})}$ hold, 
where $\widehat{K}$ and $D_1'$ are independent of $\NCM$. 
Therefore, $D_\ast(2^{k+1}) \leq 1 + (1 + c\widehat{K}) D_1' \omega_{k+1}$ holds. 
In this case, we can also obtain a positive constant $c_3$ such that $\log D_\ast(2^{k+1}) \leq c_3 + \log \omega_{k+1}$ holds.  
Consequently, if $\NCM$ satisfies the Brjuno condition, then, 
\[
\sum_{k \geq 1} \frac{\log D_\ast(2^{k+1})}{2^k} < \infty. 
\]
\end{proof}

\subsection{Ueda's lemma and Donin type estimate for the $L^\infty$-norm}\label{Linfinity}
In this subsection, as same as above, let $C$ be a compact Riemann surface and $(E, h)$ be a unitary flat line bundle over $C$ with a flat fiber metric, and fix $\U^r \ (r\in[r_\ast, 1])$ as a nested covering of $C$. 
Our goal of this subsection is introducing the Kodaira--Spencer type estimate and obtaining the Donin type estimate for the norms induced from $L^\infty$-norm on $\check{C}^0(C, E)$ and $\check{C}^1(C, E)$ . 
In particular, the former statement is known as Ueda's lemma. 

For a Hermitian line bundle $(L, h)$ over $C$, we define the norms on $\check{C}^0(C, L)$ and $\check{C}^1(C, L)$ induced from $L^\infty$-norms. 
For a \v{C}ech $0$-cochain $u = \{(U_j^r, u_j)\} \in \check{C}^0(C, L)$, we define the norm of $u$ by
\[
\|u\|_{L^\infty, \,\U^r} \coloneqq \max_{j} \|u_j\|_{L^\infty, U_j^r} = \max_j \sup_{U_j^r} |u_j|_h
\]
and  for \v{C}ech 1-cochain $f = \{(U_{jk}^r, f_{jk})\} \in \check{C}^1(C, L)$, we define the norm of $f$ by
\[
\|f\|_{L^\infty, \,\U^r} \coloneqq \max_{j, k} \|f_{jk}\|_{L^\infty, U_j^r} = \max_{j, k} \sup_{U_{jk}^r} |f_{jk}|_h. 
\]

\begin{proposition}[{\cite[Lemma 4]{U}}]\label{Ueda}
Let $C$ be a compact Riemann surface with a unitary flat line bundle $(E, \hflat)$ and  $\U^r = \{U_j^r\} (r \in [r_\ast, 1])$ a nested finite covering of $C$. 
There exists a positive constant $K_\infty$ which is independent of $E$ and all $r \in[r_\ast, 1]$, and satisfies the following:
for any $f\in\check{B}^1(\mathcal{U}^{r}, E)$, there exists $u \in \check{C}^0(\mathcal{U}^{r}, E)$ which satisfies $\delta u  = f$, and  
\[
\|u\|_{L^\infty, \,\U^r}  \leq \frac{K_\infty}{d(\one, E)} \|f\|_{L^\infty, \, \U^{r}}. 
\]
\end{proposition}
We can obtain the Donin type estimate from Ueda's lemma by similar way to the proof of Proposition \ref{correspondenceD}. 
\begin{proposition}\label{UedaD}
Let $C$ be a compact Riemann surface with a flat line bundle $(E, \hflat)$ and  $\U^r = \{U_j^r\} (r \in [r_\ast, 1])$ a nested finite covering of $C$. 
There exists a positive constant $D_\infty$ which is independent of $E$ and all $r \in[r_\ast, 1]$ and satisfies the following:
for any $f\in\check{B}^1(\mathcal{U}^{r_1}, E)$, there exists $v \in \check{C}^0(\mathcal{U}^{r_2}, E)$ which satisfies $\delta v  = f$, and  
\[
\|v\|_{L^\infty, \,\U^{r_2}}  \leq \frac{D_\infty}{(r_1- r_2)^2 \,d(\one, E)} \|f\|_{L^\infty, \, \U^{r_1}}. 
\]
Here $r_1$ and $r_2$ are any numbers satisfying $r_\ast < r_2 < r_1 \leq  1$. 
\end{proposition}
\begin{proof}
From Proposition \ref{Ueda}, we can obtain a positive constant $K_\infty$ such that, for a given \v{C}ech 1-coboundary $f\in\check{B}^1(\mathcal{U}^{r_1}, E)$, 
there exists a \v{C}ech 0-cochain $u = (U_j^{r_1}, u_j)\in \check{C}^0(\mathcal{U}^{r_1}, E)$ such that $\delta u = f$ and 
\begin{equation}\label{ueda}
\|u\|_{L^\infty, \,\U^r}  \leq \frac{K_\infty}{d(\one, E)} \|f\|_{L^\infty, \, \U^{r_1}} 
\end{equation}
hold. 
Let $v \in \check{C}^0(\mathcal{U}^{r_1}, E)$ be a \v{C}ech 0-cochain $\{(U_j^{r_2}, v_j)\}$ obtained by the restriction of $u_j$ on $U_j^{r_2}$ for each $j$. 
We can see the following inequality by using the mean value inequality for each $j$ and $\xi \in U_j^{r_2}$: 
\begin{equation}\label{infinityDonin1}
|v_j(\xi)| \leq \frac{1}{\mathrm{Vol}(\D(\xi, r_1 - r_2))} \ \int_{\D(\xi, r_1 - r_2)} |v_j| d\lambda_j. 
\end{equation}
Since $\D(\xi, r_1 - r_2) \subset U_j^{r_1} \subset U_j^1$ holds, one has
\[
\int_{\D(\xi, r_1 - r_2)} |v_j| d\lambda_j \leq \sup_{U_j^{r_1}} |v_j|_{\hflat} \int_{U_j^{r_1}}d\lambda_j \leq \sup_{U_j^{r_1}} |v_j|_{\hflat} \cdot  \mathrm{Vol}(U_j)
\]
By taking the supremum on $U_j^{r_2}$ to the inequality (\ref{infinityDonin1}) and the maximum over all $j$, and combining the inequality (\ref{ueda}), one has
\[
\|v\|_{L^\infty, \,\U^{r_2}}  \leq \frac{K_\infty\max_j \mathrm{Vol}(U_j^1)}{\pi(r_1- r_2)^2 \,d(\one, E)} \ \|f\|_{L^\infty, \, \U^{r_1}}. 
\]
Finally, we remark that a constant $D_\infty \coloneqq \pi^{-1}K_\infty\max_j \mathrm{Vol}(U_j^1)$ is independent of  $E$, $r_1$, and $r_2$. 

\end{proof}


 
\end{document}